\documentclass[a4paper,11pt,oneside]{article}
\usepackage[english]{babel}
\usepackage{amsthm,amssymb}
\usepackage{mathrsfs,amscd}
\usepackage[colorinlistoftodos]{todonotes}
\usepackage{makeidx}
\usepackage[colorlinks=false, linkcolor=blue]{hyperref}
\usepackage{float}
\DeclareFontFamily{OT1}{ptm}{}
\DeclareFontShape{OT1}{ptm}{m}{n} { <-> ptmr}{}
\DeclareFontShape{OT1}{ptm}{m}{it}{ <-> ptmri}{}
\DeclareFontShape{OT1}{ptm}{m}{sl}{ <->ptmro}{}
\DeclareFontShape{OT1}{ptm}{m}{sc}{ <-> ptmrc}{}
\DeclareFontShape{OT1}{ptm}{b}{n} { <-> ptmb}{}
\DeclareFontShape{OT1}{ptm}{b}{it}{ <-> ptmbi}{}
\DeclareFontShape{OT1}{ptm}{bx}{n} {<->ssub * ptm/b/n}{}
\DeclareFontShape{OT1}{ptm}{bx}{it}{<->ssub * ptm/b/it}{}

\DeclareSymbolFont{bold}{OT1}{ptm}{b}{n}
\DeclareMathAlphabet{\mathbf}{OT1}{ptm}{b}{n}
\DeclareMathAlphabet{\mathrm}{OT1}{ptm}{m}{n}

\DeclareFontFamily{OT1}{psy}{}
\DeclareFontShape{OT1}{psy}{m}{n}{ <-> s * [0.9] psyr}{}
\DeclareFontFamily{OMS}{ptm}{}
\DeclareFontShape{OMS}{ptm}{m}{n}{ <8> <9> <10> gen * cmsy }{}
\DeclareFontFamily{OMS}{cmtt}{}
\DeclareFontShape{OMS}{cmtt}{m}{n}{ <8> <9> <10> gen * cmsy }{}

\SetSymbolFont{operators}{normal}{OT1}{ptm}{m}{n}
\SetSymbolFont{operators}{bold}{OT1}{ptm}{b}{n}
\DeclareSymbolFont{emsy}{OT1}{ptm}{m}{it}
\DeclareSymbolFont{emsr}{OT1}{ptm}{m}{n}
\DeclareSymbolFont{emcmr}{OT1}{cmr}{m}{n}
\DeclareSymbolFont{emsymb}{OT1}{psy}{m}{n}
\DeclareMathSymbol a{\mathalpha}{emsy}{"61}
\DeclareMathSymbol b{\mathalpha}{emsy}{"62}
\DeclareMathSymbol c{\mathalpha}{emsy}{"63}
\DeclareMathSymbol d{\mathalpha}{emsy}{"64}
\DeclareMathSymbol e{\mathalpha}{emsy}{"65}
\DeclareMathSymbol f{\mathalpha}{emsy}{"66}
\DeclareMathSymbol g{\mathalpha}{emsy}{"67}
\DeclareMathSymbol h{\mathalpha}{emsy}{"68}
\DeclareMathSymbol i{\mathalpha}{emsy}{"69}
\DeclareMathSymbol j{\mathalpha}{emsy}{"6A}
\DeclareMathSymbol k{\mathalpha}{emsy}{"6B}
\DeclareMathSymbol l{\mathalpha}{emsy}{"6C}
\DeclareMathSymbol m{\mathalpha}{emsy}{"6D}
\DeclareMathSymbol n{\mathalpha}{emsy}{"6E}
\DeclareMathSymbol o{\mathalpha}{emsy}{"6F}
\DeclareMathSymbol p{\mathalpha}{emsy}{"70}
\DeclareMathSymbol q{\mathalpha}{emsy}{"71}
\DeclareMathSymbol r{\mathalpha}{emsy}{"72}
\DeclareMathSymbol s{\mathalpha}{emsy}{"73}
\DeclareMathSymbol t{\mathalpha}{emsy}{"74}
\DeclareMathSymbol u{\mathalpha}{emsy}{"75}
\DeclareMathSymbol v{\mathalpha}{emsy}{"76}
\DeclareMathSymbol w{\mathalpha}{emsy}{"77}
\DeclareMathSymbol x{\mathalpha}{emsy}{"78}
\DeclareMathSymbol y{\mathalpha}{emsy}{"79}
\DeclareMathSymbol z{\mathalpha}{emsy}{"7A}
\DeclareMathSymbol A{\mathalpha}{emsy}{"41}
\DeclareMathSymbol B{\mathalpha}{emsy}{"42}
\DeclareMathSymbol C{\mathalpha}{emsy}{"43}
\DeclareMathSymbol D{\mathalpha}{emsy}{"44}
\DeclareMathSymbol E{\mathalpha}{emsy}{"45}
\DeclareMathSymbol F{\mathalpha}{emsy}{"46}
\DeclareMathSymbol G{\mathalpha}{emsy}{"47}
\DeclareMathSymbol H{\mathalpha}{emsy}{"48}
\DeclareMathSymbol I{\mathalpha}{emsy}{"49}
\DeclareMathSymbol J{\mathalpha}{emsy}{"4A}
\DeclareMathSymbol K{\mathalpha}{emsy}{"4B}
\DeclareMathSymbol L{\mathalpha}{emsy}{"4C}
\DeclareMathSymbol M{\mathalpha}{emsy}{"4D}
\DeclareMathSymbol N{\mathalpha}{emsy}{"4E}
\DeclareMathSymbol O{\mathalpha}{emsy}{"4F}
\DeclareMathSymbol P{\mathalpha}{emsy}{"50}
\DeclareMathSymbol Q{\mathalpha}{emsy}{"51}
\DeclareMathSymbol R{\mathalpha}{emsy}{"52}
\DeclareMathSymbol S{\mathalpha}{emsy}{"53}
\DeclareMathSymbol T{\mathalpha}{emsy}{"54}
\DeclareMathSymbol U{\mathalpha}{emsy}{"55}
\DeclareMathSymbol V{\mathalpha}{emsy}{"56}
\DeclareMathSymbol W{\mathalpha}{emsy}{"57}
\DeclareMathSymbol X{\mathalpha}{emsy}{"58}
\DeclareMathSymbol Y{\mathalpha}{emsy}{"59}
\DeclareMathSymbol Z{\mathalpha}{emsy}{"5A}
\DeclareMathSymbol{\bullet}{\mathalpha}{emsymb}{"B7}
\DeclareMathSymbol{\regis}{\mathalpha}{emsymb}{"D2}
\def\Bullet{\leavevmode\unkern{$\m@th\bullet$}\kern.32em\ignorespaces}
\def\Regis{\leavevmode\raise.5ex\hbox{$\m@th\regis$}}
\DeclareMathSymbol +{\mathbin}{emcmr}{`+}
\DeclareMathSymbol ={\mathrel}{emcmr}{`=}
\DeclareMathSymbol{\Gamma}{\mathalpha}{emcmr}{"00}
\DeclareMathSymbol{\Delta}{\mathalpha}{emcmr}{"01}
\DeclareMathSymbol{\Theta}{\mathalpha}{emcmr}{"02}
\DeclareMathSymbol{\Lambda}{\mathalpha}{emcmr}{"03}
\DeclareMathSymbol{\Xi}{\mathalpha}{emcmr}{"04}
\DeclareMathSymbol{\Pi}{\mathalpha}{emcmr}{"05}
\DeclareMathSymbol{\Sigma}{\mathalpha}{emcmr}{"06}
\DeclareMathSymbol{\Upsilon}{\mathalpha}{emcmr}{"07}
\DeclareMathSymbol{\Phi}{\mathalpha}{emcmr}{"08}
\DeclareMathSymbol{\Psi}{\mathalpha}{emcmr}{"09}
\DeclareMathSymbol{\Omega}{\mathalpha}{emcmr}{"0A}
\DeclareMathSizes{7.6}{8}{6}{5}
\def\`#1{{\accent"12 #1}}            
\chardef\J="11
\chardef\AA="C8                      
\chardef\gbp="A3                     
\chardef\TIL="81                     
\chardef\endash="B1
\chardef\emdash="D0
\chardef\pourmille="BD               
\chardef\aoben="E3                   
\chardef\ooben="EB                   
\def\S{\leavevmode\unkern{\char"A7}\kern.1em\ignorespaces}
\DeclareMathAccent{\dot}{\mathalpha}{operators}{"C7} 

\newtheorem{pro}{Proposition}
\newtheorem{te}{Theorem}

\newtheorem{re}{Remark}

\newtheorem{Qu}{Question}
\title{\textbf{Symmetric ribbon disks}}
\author{Paolo Aceto}
\date{}
\begin{document} 
\maketitle
\begin{abstract}
 We study the ribbon disks that arise from a symmetric union presentation of a ribbon knot.
 A natural notion of symmetric ribbon number $r_S(K)$ is introduced and compared
 with the classical ribbon number $r(K)$. We show that the difference $r_S(K)-r(K)$
 can be arbitrarily large by constructing an infinite family of ribbon knots $K_n$ 
 such that $r(K_n)=2$ and $r_S(K_n)>n$.
 The proof is based on a particularly simple description of symmetric unions in terms
 of certain band diagrams which leads to an upper bound for the Heegaard genus of their branched
 double covers.
 \end{abstract}
\begin{section}{Introduction}
Symmetric unions (see Section \ref{S2} for definitions) were first introduced 
by Kinoshita and Terasaka in \cite{KTbis}.  Christoph Lamm began a systematic investigation 
in \cite{Lamm:2000}. One of the main questions which motivate the study of symmetric 
unions is the following.
\begin{Qu}\label{mainquestion}
Does every ribbon knot admit a symmetric union presentation? 
\end{Qu}
In \cite{Lamm:2000} it is shown that all ribbon knots up to ten crossings have a symmetric union 
presentation. In \cite{Lamm:2006} the result is extended to all 2-bridge ribbon knots. In 
\cite{EisermannLamm:Eq} Eisermann and Lamm introduced a natural notion of \emph{symmetric 
equivalence} between symmetric union presentations providing a list of symmetric Reidemeister-like
moves. This notion is further investigated in \cite{EisermannLamm:SymJones} where a refined
version of the Jones polynomial is adapted to this setting.

One way to measure the complexity of a given ribbon knot $K$ is to consider its \emph{ribbon number}
$r(K)$ which is the minimal number of ribbon singularities among all ribbon disks spanning $K$. 
Every symmetric union presentation has an obvious symmetric ribbon disk associated with it, therefore
it makes sense to define the \emph{symmetric ribbon number} $r_{S}(K)$ of a given ribbon knot
$K$. Set $r_S(K)=\infty$ if $K$ does not admit a symmetric union presentation. Keeping Question
\ref{mainquestion} in mind it is natural to compare the ribbon number and its symmetric version 
for a given knot.
It turns out that the complexity of a symmetric union presentation (as measured by the 
symmetric ribbon number) can be arbitrarly large even for knots with $r(K)=2$.
\begin{te}\label{maintheorem}
 For each positive integer $n$ there exists a ribbon knot $K_n$ such that $r(K_n)=2$ and
 $r_S(K_n)>n$.
\end{te}

We do not know if any of the ribbon knots in the family $\{K_n\}$ admits a symmetric union 
presentation and we suspect that many of them do not. These knots are constructed as satellites
of the simplest non trivial ribbon knot, i.e. the connected sum of a trefoil knot with its mirror
image. The proof is based on the fact that for any ribbon knot $r_s(K)$ gives un upper bound
for the Heegaard genus of the branched double cover along $K$. We also show that there is a similar
inequality involving the \emph{free genus} of $K$.

In Section \ref{S2} we briefly review basic properties of symmetric unions and we establish
the inequalities we need. In Section \ref{S3} we introduce the family $\{K_n\}$ and give an estimate
for the Heegaard genus of their branched double covers which is needed to complete the 
proof of Theorem \ref{maintheorem}.
\end{section}

\begin{section}{Symmetric Unions, Branched Double Covers and the Free Genus}\label{S2}
A \emph{symmetric union presentation} of a knot is a diagram obtained in the following way. We start
with the symmetric diagram of a connected sum between a knot and its mirror image, and then we add
crossings on the axis of symmetry.
See Figure \ref{ex} for some examples taken from \cite{Lamm:2000}. 
See also \cite{Lamm:2000} for a formal definition.
\begin{figure}[h!]
	    \centering
	 {\qquad\includegraphics[scale=0.2]{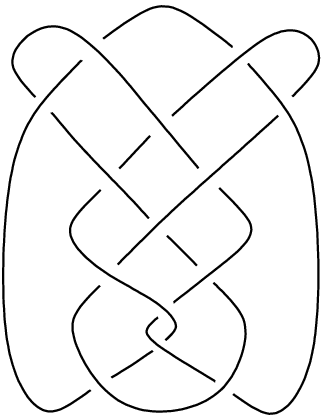}\quad}
  {\qquad\includegraphics[scale=0.2]{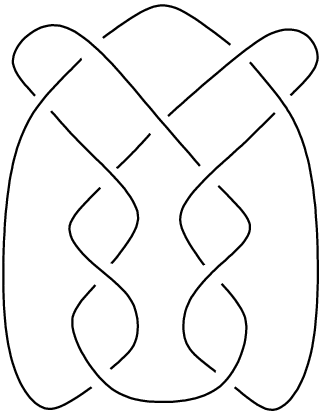}\quad}
 {\quad\includegraphics[scale=0.2]{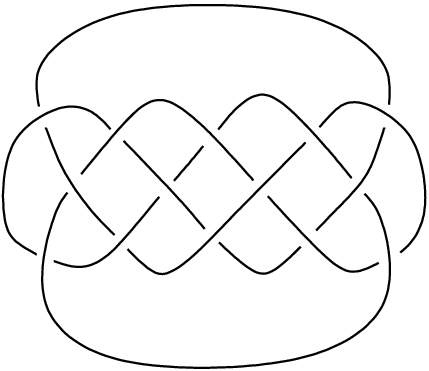}\quad}
 \caption{Symmetric union presentations for the knots $9_{41}$, $10_{22}$ and $10_{123}$.}
	\label{ex}
	    \end{figure}
The obvious reflection will fix the knot except near crossings
on the axis which are necessarly reversed. More generally one can talk about
symmetric diagrams, these will represents links that bounds ribbon surfaces consisting of 
annuli, Mobius strips and disks. Each symmetric union diagram describes an obvious ribbon disk
spanning the given knot, see Figure \ref{simdisk} for an example.

\begin{figure}[H]
	    \centering
	\includegraphics[scale=0.3]{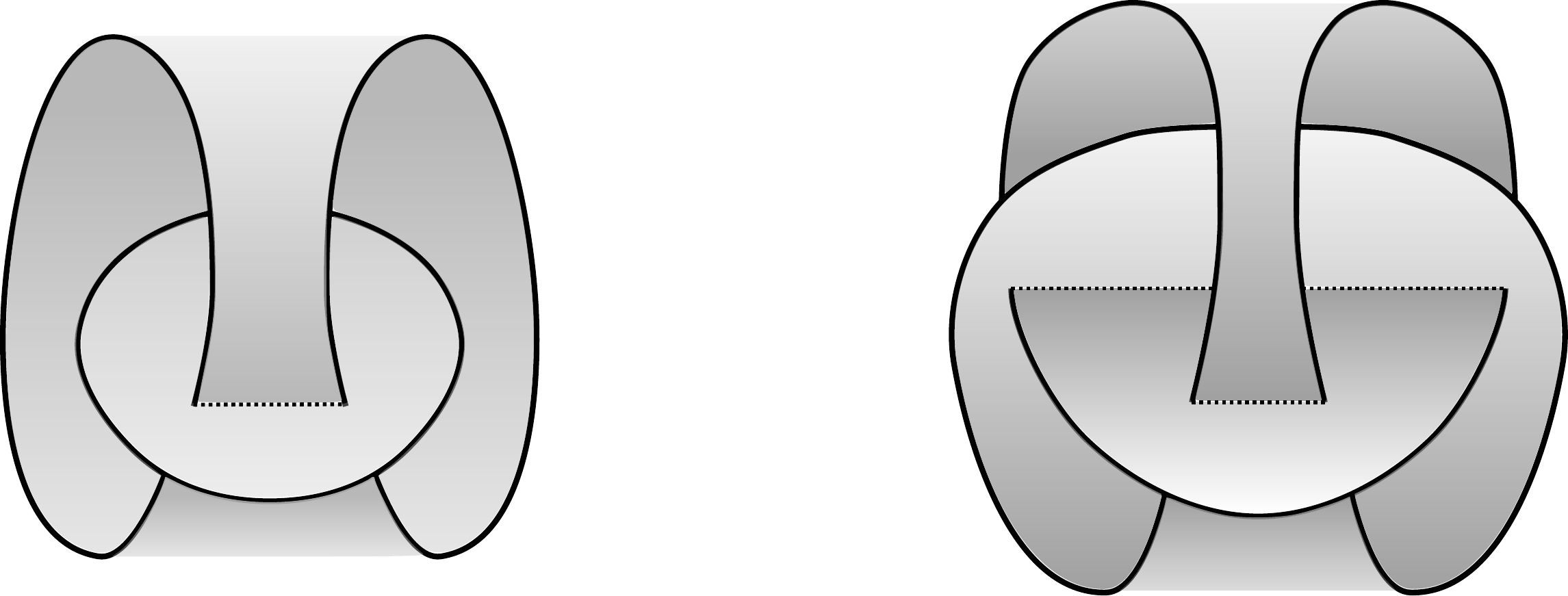}
	\caption{On the left: a symmetric link diagram bounding a ribbon surface 
	consisting of a disk and an annulus. On the right: a symmetric union diagram and the
	corresponding symmetric ribbon disk.}
	\label{simdisk}
	    \end{figure}
	    
This disk is obtained by joining symmetric points 
with a straight line (see \cite{EisermannLamm:Eq}). We call such disks
\emph{symmetric ribbon disks}.
Recall from 
\cite{Eisermann:2009} that every ribbon disk (and more generally every ribbon surface) in $S^3$
admit a diagram, called a \emph{band diagram}, made up from the elementary pieces shown in Figure \ref{banddiagrams}.
\begin{figure}[H]  
  {\quad\qquad\includegraphics[scale=0.8]{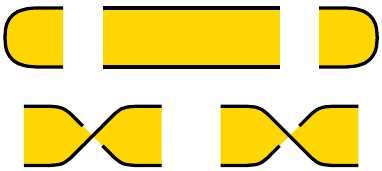}\quad}
  {\qquad\includegraphics[scale=0.8]{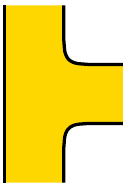}\quad}
 {\quad\includegraphics[scale=0.8]{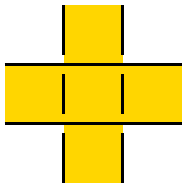}\quad}
 {\qquad\includegraphics[scale=0.8]{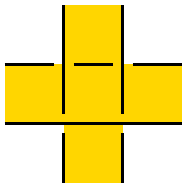}}
  \caption{Elementary pieces of band diagrams}
  \label{banddiagrams}
\end{figure}
The second picture represents a \emph{junction}.
The third picture represents a \emph{band crossing} and the fourth one a ribbon singularity.
Our first observation
is that symmetric ribbon disks admit a nice combinatorial characterization in terms of 
band diagrams. 
\begin{pro}\label{keybservation}
 Every symmetric ribbon disk admits a band diagram without any band crossings nor junctions. 
 Conversely, every ribbon disk that has a band diagram without band crossings or junctions 
 is a symmetric ribbon disk.
\end{pro}
\begin{proof}
First note that each crossing on the axis corresponds to a half twist of a band which is 
a portion of the symmetric ribbon disk. As a first step
we momentarily remove all these twists by smoothing crossings on the axis as shown in 
Figure \ref{smooth}.
\begin{figure}[h!]  
\centering
   \includegraphics[scale=0.5]{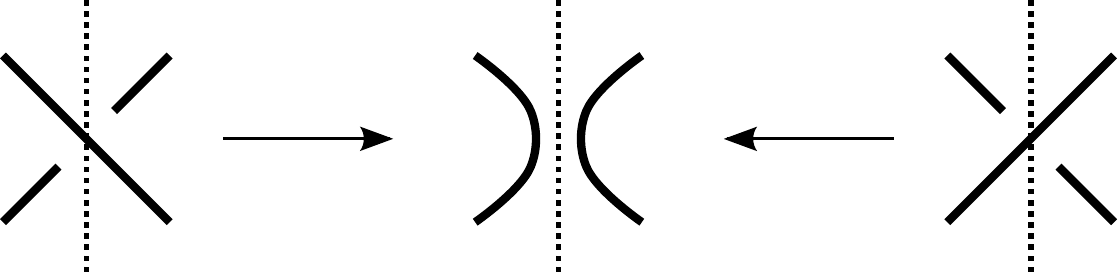}
  \caption{How to smooth crossings on the axis so that they correspond to the removal of half twists
  in the symmetric ribbon disk.}
  \label{smooth}
\end{figure}

Recall that the disk
is obtained by joining symmetric points with a straight line orthogonal to the symmetry plane. 
It follows that the projection of the disk onto the symmetry plane consists of an arc 
with selfintersections (double points). For example by projecting 
the symmetric ribbon disk in Figure \ref{simdisk} on its symmetry plane we get the curve in Figure 
\ref{selfint}.
\begin{figure}[h!]  
\centering
   \includegraphics[scale=0.4]{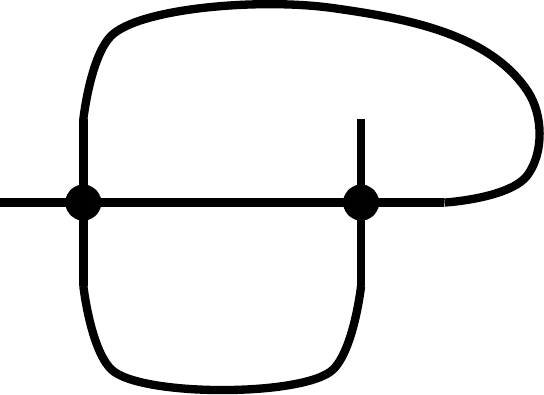}
  \caption{A selfintersecting curve resulting from the projection of a symmetric ribbon disk onto 
  its symmetry plane.}
  \label{selfint}
\end{figure}
Each smooth point on this curve corresponds to a properly embedded interval in the ribbon disk while
double points correspond to ribbon singularities.
 The desired band diagram may be obtained by slightly perturbing this projection as follows.
 First we modify the projection in a neighbourhood of each double point as shown in Figure 
 \ref{perturbing}. 
 \begin{figure}[h!]  
\centering
   \includegraphics[scale=1]{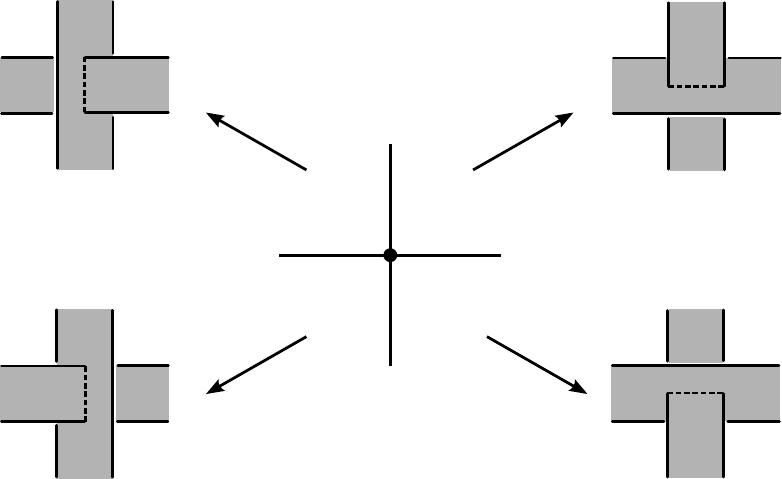}
  \caption{How to perturb the projection near a double point. }
  \label{perturbing}
\end{figure}
 The perturbation is obtained by rotating each band on its core, this may be done
 in two different ways which correspond to the two rows in Figure \ref{perturbing}. Now we slightly
 perturb the rest of the projection which appears as a collection of arcs. Each arc will appear as
 a band after a rotation on its core. Finally note that all these local choices for the rotation of
 bands can produce band twists. Specifically twists will arise each time two adiacent bands
 are rotated in opposite directions.
 
 Now the original crossings on the axis can be restored. 
 Each crossing on the axis correspond to a half twist of a band of the ribbon disk.
 This band may be located in the projection we just obtained and the half twist can be reinserted.
 The band diagram obtained in this way has no band crossings and no
 junctions.
 
 Conversely suppose we are given a band diagram of a ribbon disk without band crossings and 
 without junctions. Near each ribbon singularity we may change the projection to obtain
 a pair of arcs intersecting trasversely in a single point (see Figure \ref{perturbing}). We may do the
 same for the bands connecting these ribbon singularities. Each band will look like an arc
 outside small regions corresponding to band twists. Near each ribbon singularity we may choose 
 a plane parallel to the projection plane that cuts each band along its core. 
 We may assume that all these planes coincide because all ribbon singularities are connected 
 by band in planar fashion. We may also assume that the plane cuts each band connecting two 
 ribbon singularities along its core (except near band twists). 
 All these adjustments may be carried out without altering our projection.
 Note that the reflection across this plane fixes the knot except near the band twists.
 A symmetric union diagram can be obtained
 by projecting the knot onto any plane orthogonal to symmetry plane.
 
\end{proof}
\begin{re}\label{re}
 The band diagram obtained in the above proof has the same number of ribbon singularities
as the original symmetric ribbon disk.
\end{re}
\begin{re}\label{rebis}
 Proposition \ref{keybservation} holds in the more general context of symmetric diagrams.
 Every symmetric link bounds a ribbon surface consisting of disks, annuli and Mobius bands.
 A suitable projection can be chosen so that this surface can be described via a 
 band diagram without band crossings and without junctions. The proof is the same as that
 of Proposition \ref{keybservation}.
\end{re}

Before we explore some consequences of Proposition \ref{keybservation} we introduce some notation.
Given a knot $K$ we indicate by $g_F(K)$ the \emph{free genus} of $K$, i.e. the minimal genus
among all Seifert surfaces spanning $K$ whose complement have free fundamental group (these
are sometimes called \emph{regular} Seifert surfaces). We denote 
by $\Sigma(K)$ the branched double cover of $S^3$ branched along $K$. Finally for any closed
orientable 3-manifold $Y$ we denote by $g_H(Y)$ its Heegaard genus.
\begin{pro}\label{inequalities}
 Let $K$ be a ribbon knot, and let $D$ be a symmetric ribbon disk which minimizes 
 the number of ribbon singularities.
 \begin{enumerate}
  \item $S^3\setminus D$ has free fundamental group and its rank equals $r_S(K)$
  \item $r_S(K)\geq \frac{1}{4}(g_H(\Sigma(K))+1)$
  \item $r_S(K)\geq g_F(K)$
 \end{enumerate}
\end{pro}
\begin{proof}
 The first assertion is clear from the proof of Proposition \ref{keybservation} and Remark \ref{re}. 
 \begin{figure}[h!]
	    \centering
	\includegraphics[scale=0.8]{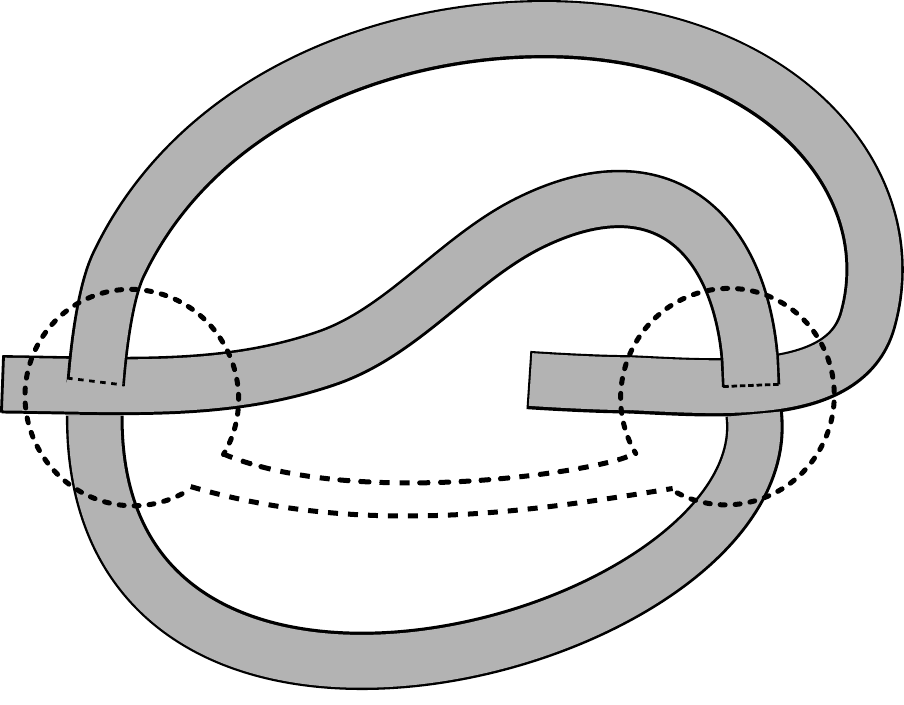}
	\caption{A decomposition of $(S^3,K)$ that lifts to a Heegaard decomposition of $\Sigma(K)$.}
	\label{exbis}
	    \end{figure}
 Choose a band diagram for $D$ without band crossings and without junctions. We can enclose each ribbon singularity
 in a ball whose boundary intersects the knot in exactly eight points. Now we can connect these
 balls with unknotted tubes disjoint from $D$, in this way we obtain a decomposition of
 $S^3$ into two balls $B_1$ and $B_2$. Each pair $(B_i,B_i\cap K)$ is homeomorphic to the standard
 ball containing a trivial braid, therefore this decomposition lifts to a Heegaard decomposition
 for $\Sigma(K)$ of genus $3r_S(K)$ (see Figure \ref{exbis} for an example). The inequality follows.
 
 Let us fix again a band diagram for $D$ without band crossings and without junctions
 and an orientation for $K$. 
 Near each ribbon singularity there are four possible oriented configurations as shown in 
 Figure \ref{orientedconfigurations}. 
 \begin{figure}[h!]
	    \centering
	\includegraphics[scale=1.4]{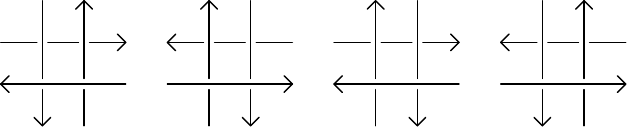}
	\caption{The four possible oriented configurations of a ribbon singularity}
	\label{orientedconfigurations}
	    \end{figure}
 By flipping the horizontal band if necessary
  we may assume that only the first two cases occur. Now assume, momentarly, that there are
  no band twists in our diagram. In this situation
 the Seifert algorithm will produce a regular Seifert surface $S$ whose genus is precisely $r_S(K)$.
 To see this note that near each ribbon singularity the oriented resolution gives four arcs 
 corresponding to the four bands connected to the ribbon singularity. All but two of the 
 Seifert circles will appear as an arc on two ribbon singularities (once these are appropriately
 resolved). Therefore the number of Seifert circles of $S$ is $2r_S(K)+1$. This means  that
 the Euler characteristic is $\chi(S)=2r_S(K)+1-4r_S(K)=1-2r_S(K)$ from which we obtain
 $g(S)=r_S(K)$. It is easy to check that band twists do not alter the computation we made above
 because they do not change the Seifert graph associated to $S$. 
 \end{proof}
 \end{section}
\begin{section}{The Family $\{K_n\}$}\label{S3}
In this section we will construct an infinite family of ribbon knots $\{K_n\}$ such that 
$r(K_n)=2$ and $g_H(\Sigma(K_n))\geq n$. This fact together with the second assertion of 
Proposition \ref{inequalities} will conclude the proof o\label{lemma}f Theorem \ref{maintheorem}.
The idea is to start with a very simple ribbon knot and then replace a piece of unknotted band
with a knotted one, this operation will change the knot and its ribbon disk 
without changing its ribbon number. 
See Figure \ref{ribsat} for an example. 

First we recall some basic terminology and facts from \cite{Lickorish:1981}. By a \emph{tangle} we
mean a pair $(B,t)$ where $B$ is a 3-ball and $t$ is a pair of properly embedded arcs in $B$. Two 
tangles $(B,t),(B',t')$ are \emph{equivalent} if there is an homeomorphism  of pairs from $(B,t)$
to $(B',t')$. A tangle is said to be \emph{untangled} is it is equivalent to the trivial tangle 
(these are usually called \emph{rational} tangles). A tangle is \emph{locally unknotted} if every
2-sphere which meets $t$ transversely in two points bounds a ball which meets $t$ in an unknotted 
arc. Finally a locally unknotted tangle which is not untangled is said to be \emph{prime}. 
Given two tangles $(B,t)$, $(B',t')$ and a homeomorphism 
$\varphi:(\partial B,\partial t)\rightarrow (\partial B',\partial t')$ a link can be obtained
by identifying the boundaries of the two tangles via 
$\varphi$. We say that such a link is obtained as the 
\emph{sum of the tangles} $(B,t)$ and $(B',t')$. We will also need the notion of \emph{partial sum}
between tangles. This operation consists in identifying a $(\textrm{disk},\textrm{pair of points})$
pair in the boundary of a tangle with a similar pair in the boundary of another tangle so that
the result is still a tangle.
\begin{te}\label{lick}\cite{Lickorish:1981}
 Let $(B,t)$ be a prime tangle. Its branched double cover $Y$ is irreducible and its boundary
 $\partial Y$ is incompressible.
\end{te}
We recall from \cite{Eudave:2001} the following theorem.
\begin{te}\label{bound}
There exists a constant $C_g$ such that for every closed, orientable and irreducible 
3-manifold $M$ of Heegaard genus $g$ and every collection $T_1,\dots,T_k$, $k>C_g$ of disjoint 
incompressible tori in $M$ at least two of them $T_i$ and $T_j$ are parallel. 
\end{te}
Let $T_1$ and $T_2$ be the tangles depicted in Figure \ref{tangles}.
Let $K_n$ be the knot obtained as an iterated partial sum of the tangle $T_1$ and $n$ copies 
of the tangle $T_2$ as depicted in Figure \ref{ribsat}. 

As is clear from the picture each 
$K_n$ is a nontrivial knot bounding a ribbon disk with two ribbon singularities. Moreover
it is easy to show that every ribbon disk with less than two ribbon singularities is bounded
by the trivial knot (for these and other considerations on the ribbon number see for instance
\cite{Mizuma:2005}). We conclude that $r(K_n)=2$ for each $n>0$.
 \begin{figure}[h!]
	    \centering
	\includegraphics[scale=0.7]{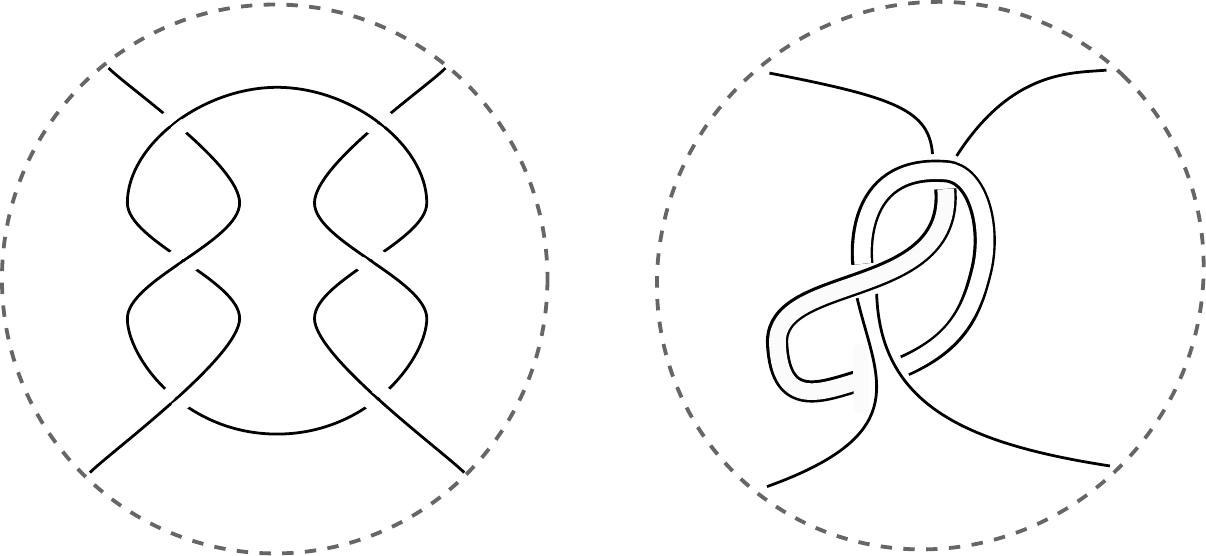}
	\caption{Two prime tangles.}
	\label{tangles}
	    \end{figure}
\begin{pro}
 For each positive integer $k$ there exists $n_k$ such that
 $$
 g_H(\Sigma(K_{n_k}))> k
 $$
\end{pro}
\begin{proof}
It is shown in \cite{Lickorish:1981} that the tangles $T_1$ and $T_2$ are prime and that summing
prime tangles gives prime knots. By \cite{KT} a link in $S^3$ is prime if and only if
its branched double cover is irreducible. It follows that each $\Sigma(K_n)$ is an irreducible 
3-manifold.
\begin{figure}[h!]
	    \centering
	\includegraphics[scale=0.6]{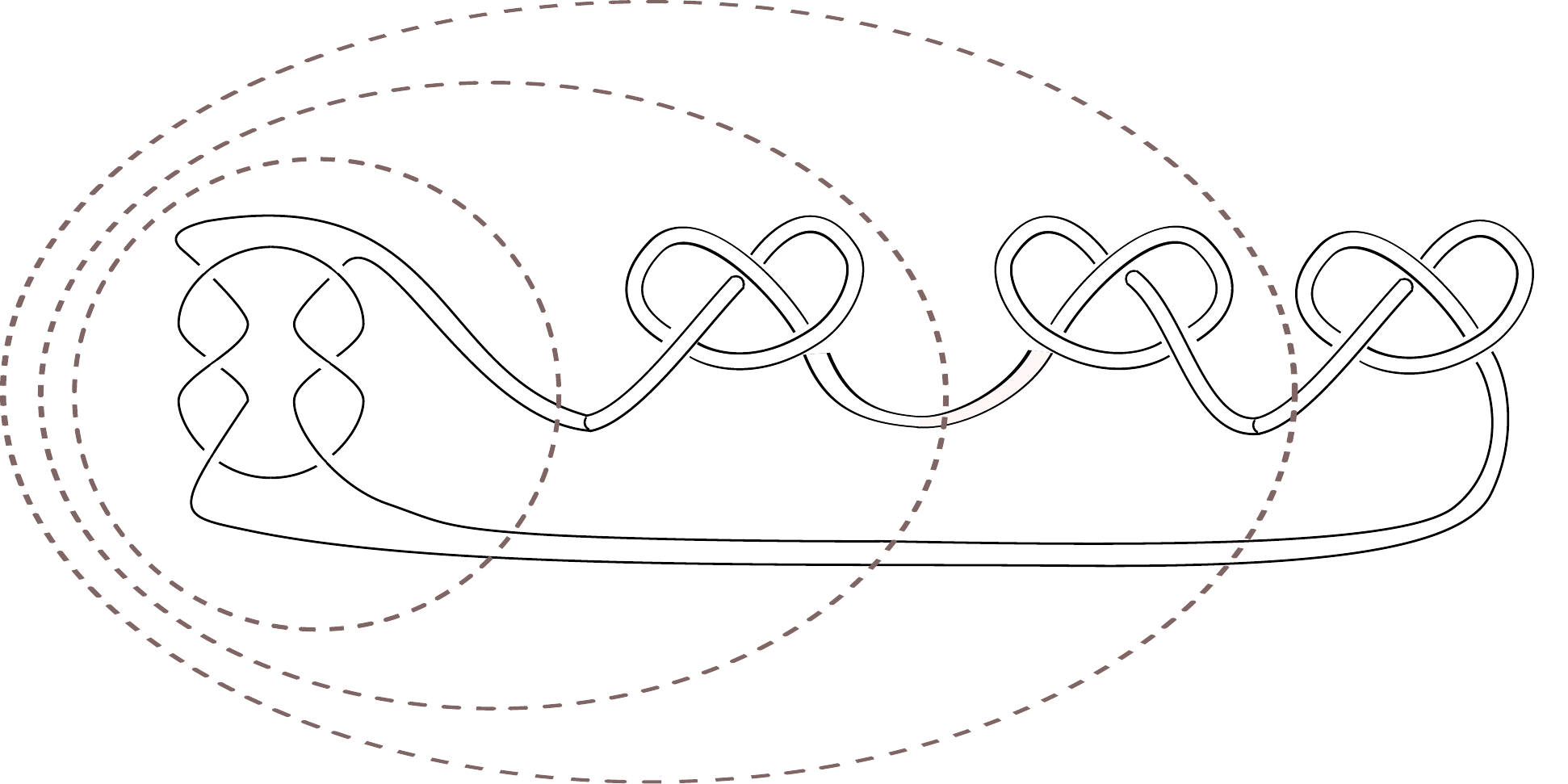}
	\caption{The knot $K_3$. Each sphere lifts to an incompressible torus 
	in the branched double cover}
	\label{ribsat}
	    \end{figure}
By Theorem \ref{lick} each sphere depicted in Figure \ref{ribsat} lifts to an 
incompressible torus in $\Sigma(K_n)$. We have $n$ disjoint incompressible tori, in order to apply
Theorem \ref{bound} we only need to show that these tori are pairwise non parallel. 
Choose a pair of tori and assume by contradiction
that they are parallel. By Dehn filling each boundary component we would obtain a Lens space. This
filling may be chosen so that it corresponds to the sum of trivial tangles in $S^3$ which gives
a cable of a connected sum of trefoil knots. Since this knot is not a 2-bridge knot 
its branched double cover cannot be a Lens space.
\end{proof}
\end{section} 
\section*{Acknowledgements}
I wish to thank Bruno Martelli for helpful conversations, my advisor Paolo Lisca, Christoph Lamm
and Micheal Eisermann for letting me use some of their pictures and Giulia Cervia
for her help in drawing pictures.

{}

\end{document}